\documentclass[11pt]{article}

\usepackage[margin=1in]{geometry}
\usepackage{amsmath,amsfonts,amssymb,amsthm}
\usepackage{graphicx}
\usepackage[authoryear]{natbib}
\usepackage[colorlinks,citecolor=blue,urlcolor=blue]{hyperref}
\newcommand{\doi}[1]{doi: \href{https://doi.org/#1}{\nolinkurl{#1}}}

\newtheorem{theorem}{Theorem}
\newtheorem{corollary}{Corollary}
\newtheorem{lemma}{Lemma}
\theoremstyle{definition}
\newtheorem{definition}{Definition}

\theoremstyle{remark}
\newtheorem{remark}{Remark}
\newenvironment{bitem}[1]{\par\medskip\noindent\textbf{#1.}\ \itshape}%
{\par\medskip}

\begin{document}

\title{A New Look at the Classical Estimation Problem}

\author{Paul W. Vos\thanks{Department of Public Health, East Carolina University,
Greenville, North Carolina. Email: \texttt{vosp@ecu.edu}. ORCID: 0000-0001-9996-5627.}}

\date{November 2026}

\maketitle

\begin{abstract}
Bahadur's \emph{Lectures on the Theory of Estimation} develop the
classical theory of point estimation inside the geometry of
Hilbert space, and they record with unusual
honesty where the theory strains: the locally best unbiased estimate
depends on the parameter, a two-point
parameter space yields an estimate Bahadur calls absurd, the odds
in binomial sampling have no unbiased estimate, and the virtues
of maximum likelihood enter as heuristics and remain heuristics. We
present a subset of the lectures, in Bahadur's notation and
development, and at each strain make one small modification: for
each value in the sample space, an
estimate $\tau$ becomes a function on the parameter space
rather than a point in it. The result is the continuum of null
hypotheses that Fisher, in 1955, identified with the theory of
estimation; we call this continuum of significance tests
\emph{generalized estimation}, and $\tau$ a generalized
estimator.
Bahadur's own definition of an estimate,
square-integrable at every distribution in the family, already
supplies the domain. The
payoffs are tracked lecture by lecture: estimators that exist at
boundary samples where point estimates do not; an elementary lemma
showing that no pointwise criterion admits a uniformly optimal
estimator, which explains why admissibility, minimaxity, Bayes
averaging, and unbiasedness arose as responses; assessment by
information, $\Lambda(\tau)$, with the score attaining the Fisher
information bound
uniformly by a three-line argument; Cram\'{e}r--Rao attainment and
sufficiency recovered as equality cases of that bound under two maps
from point estimators to generalized estimators; and the maximum
likelihood heuristics converted into exact statements about the
score. Nothing classical is overturned; the classical apparatus is
explained by generalized estimation---Fisher's characterization of
estimation as a continuum of significance tests.
\end{abstract}


\section{Introduction}\label{sec:intro}

R.~R. Bahadur's \emph{Lectures on the Theory of Estimation}
\citep{Bahadur2002} present the classical theory of point estimation
with unusual economy: the entire development---unbiased estimation,
the information inequality, the Bhattacharya%
\footnote{We follow Bahadur's spelling; the name is now more commonly
transliterated Bhattacharyya.} bounds, sufficiency, the maximum
likelihood heuristics---is carried out inside the geometry of the
Hilbert spaces $L^{2}(P_{\theta})$, using little more than orthogonal
projection. The lectures are also unusually honest. At each stage
Bahadur records exactly where the theory
strains: the locally best unbiased estimate ``often depends on
$\theta$, and this is the problem in practice''; a two-point parameter
space produces an unbiased estimate he calls absurd; the odds in
binomial sampling have no unbiased estimate at all, and the plug-in
estimate has infinite expectation; the celebrated properties of
maximum likelihood enter as heuristics and, within the lectures,
remain heuristics. These asides are not blemishes. They mark, with
some precision, the places where the classical formulation asks the
wrong object to do the work.

This paper presents a subset of Bahadur's lectures and, at each of
those places, makes a single small modification. For a family of
distributions $\left\{P_\theta:\theta\in\Theta\right\}$, Bahadur
defines an
estimate as a measurable function $t$ on the sample space $S$
satisfying
$E_{\theta}(t^{2})<\infty$ for \emph{every} $\theta$ in the parameter
space---membership, that is, in every Hilbert space
$L^{2}(P_{\theta})$ at once. The family is thus already present in the
definition; the classical theory admits it at the door and then asks
it to wait outside. The modification is to let it in: an estimator
becomes a function $\tau$ on $S\times\Theta$, so that the estimate
at $s$,
$\tau(s,\cdot)$, is a \emph{function on the parameter
space} rather than a point in it.
The cost is one extra argument. The payoff, which the paper develops, is a theory that
is simpler where the classical theory is strained: estimators that
exist at samples where point estimates do not; local dependence on
$\theta$ transformed from ``the problem in practice'' into the
defining property of the inferential object; and an assessment
criterion---information rather than variance---under which a uniform
optimality theorem is available and the auxiliary apparatus of
unbiasedness, admissibility, and minimaxity is no longer needed to
repair the absence of one.

The look is new; the inferential object is not. Fisher, writing in
1955 on the theory of estimation, described it in exactly these
terms: ``It may
be added that in the theory of estimation we consider a continuum of
hypotheses each eligible as null hypothesis, and it is the aggregate
of frequencies calculated from each possibility in turn as true
\ldots{} which supply the likelihood function''
\citep[p.~73]{Fisher:1955}. A generalized estimator is that
continuum, made an explicit object of the theory and given an
assessment criterion of its own.

The assessment criterion deserves emphasis, because the richer object
turns out to be the \emph{simpler} one to assess. Classical risk is
pointwise: $R_{t}(\theta)$ is computed from the sampling distribution
of $t$ under $P_{\theta}$ alone. An elementary lemma
(Section~\ref{sec:classical}) shows that no pointwise criterion admits
a uniformly optimal estimator, which is why the classical theory must
change the question---averaging risk, comparing maxima, or restricting
to the unbiased class---before optimality can be discussed at all. The
information of a generalized estimator is not pointwise: it measures
the average rate at which $\tau$ responds as the distribution moves
within the family, and it supports a uniform optimality theorem whose
proof is a few lines.
The engine is a single differentiation identity, stated in
Section~\ref{sec:score}, that Bahadur's own computations already run
on.
That so simple a fact characterizes efficiency is itself a
consequence of formulating estimation in terms of generalized
estimators.

The framework presented here is developed in \citet{VosWu2025} in the
language of information geometry
\citep{Amari1985,AmariNagaoka2000}, and applied to the James--Stein
paradox in \citet{Vos2026-mse}. Neither source is presupposed. The aim
of this paper is to make the framework available to readers at home in
the classical development---which Bahadur's lectures state more
carefully than any comparable source---by modest detours from the
lectures themselves, in Bahadur's own notation, and without
differential-geometric terminology: what is needed is one space of
square-integrable functions carrying a family of inner products, and
that much is already in Lecture 1.

We follow Bahadur's notation and development closely, and the scope
is correspondingly his. The parameter is
scalar throughout; the multiparameter case (his Chapter 6) and the
asymptotic theory (his Chapter 8) are left out. Bayes estimation, to
which Bahadur devotes Lectures 6--10, appears only in a single remark,
because in the present reading Bayes averaging is one of the classical
responses to a structural impossibility rather than a topic to be
developed. Decision-theoretic machinery is not imported. The running
example is Bahadur's own Example 2---Bernoulli sampling under several
stopping rules---which he threads through the lectures on information,
sufficiency, and unbiased estimation, and which supplies both the
boundary samples and the odds at which the classical and
generalized theories visibly part company.

The paper proceeds as follows. Section~\ref{sec:classical} states the
classical estimation problem following Lecture 5, introduces the
generalized estimator, and proves the impossibility lemma that
organizes the classical responses. Section~\ref{sec:unbiased} presents
unbiased estimation and the spaces $W_{\theta}$ of Lectures 11--12.
Section~\ref{sec:score} presents the score, Fisher information, and
the bounds of Lectures 13--15, and introduces the information
$\Lambda(\tau)$ of a generalized estimator.
Section~\ref{sec:attainment} presents the attainment theory and the
one-parameter exponential family of Lectures 16--18, where the
equality case of the information bound recovers sufficiency.
Section~\ref{sec:heuristics} revisits the maximum likelihood
heuristics of Lectures 14 and 25--27, which become exact statements
under the modification. Section~\ref{sec:hilbert} compares Bahadur's
Hilbert spaces with the one used by generalized estimation.
Section~\ref{sec:discussion} summarizes what the modification
provides, what has been deferred to the multiparameter theory of
\citet{VosWu2025}, and where generalized estimation
is itself silent.

\section{The classical estimation problem}\label{sec:classical}

\subsection{Lecture 5: Bahadur's formulation}\label{subsec:formulation}

Lecture 5 opens with the problem that occupies the remainder of the
lectures. There is a sample space $S$ with sample point $s$, a
$\sigma$-field $\mathcal{A}$ on $S$, and a set
$\mathcal{P}=\{P_{\theta}:\theta\in\Theta\}$ of probability measures
on $\mathcal{A}$ indexed by the parameter space $\Theta$.
A real-valued estimand $g$ is given on
$\Theta$, and the problem is to estimate the actual value $g(\theta)$
from the observed $s$ and to describe the quality of the estimate.

Two standing assumptions, both Bahadur's, are recorded here. The
members of $\mathcal{P}$ are mutually absolutely continuous---they
share support---as required from Lecture 11 onward, where the
likelihood ratios $\Omega_{\delta,\theta}=dP_{\delta}/dP_{\theta}$
become the basic building blocks. The parameter space  is a connected set,
typically an open interval; Bahadur himself concludes, after the
two-point example recalled in Section~\ref{sec:unbiased}, that ``we
should restrict the estimation theory to a continuum of values (i.e.,
should have only connected $\Theta$).''

\begin{definition}[Bahadur, Lecture 5]\label{def:estimate}
An \emph{estimate} (of $g(\theta)$) is a measurable function $t$ on
$S$ such that
\begin{equation}
E_{\theta}(t^{2})=\int_{S}t(s)^{2}\,dP_{\theta}(s)<\infty
\quad\text{for all }\theta\in\Theta.\label{eq:estimate-def}
\end{equation}
\end{definition}

A note on usage: we distinguish an \emph{estimator}, a function on
the sample space, from its realized value at the observed sample
point, the \emph{estimate}. Bahadur does not make this distinction;
under our convention the object of Definition~\ref{def:estimate} is
an estimator. We retain his word when presenting his development.

The definition is stated before any criterion of quality, and it
already contains the observation on which this paper turns.
Lectures 1--4 develop the geometry of
$L^{2}(S,\mathcal{A},P)$ for a single probability measure $P$: a
Hilbert space of (equivalence classes of) square-integrable
functions, with inner product
$(f,h)=E(fh)$, in which orthogonal projections onto subspaces exist,
are unique, and coincide with conditional expectations when the
subspace consists of the square-integrable functions of a statistic.
For each $\theta$ there is such a space
$V_{\theta}=L^{2}(S,\mathcal{A},P_{\theta})$, and
condition~\eqref{eq:estimate-def} says precisely that
\begin{equation}
t\in V_{\theta}\ \text{ for all }\theta\in\Theta;
\quad\text{equivalently,}\quad
t\in\bigcap_{\theta\in\Theta}V_{\theta}.\label{eq:membership}
\end{equation}
An estimate is not an object attached to one distribution; by
definition it lives in every $V_{\theta}$ simultaneously. The family
enters the classical theory at the definition of its basic object and
is then, for the whole of the classical development, excluded from the
assessment of that object---more precisely, the classical assessment
uses the members of the family one at a time, as a bare set of
distributions, and what it never uses is their relatedness.
The modification made in
Section~\ref{subsec:modification} does no more than put
\eqref{eq:membership} to work.

\subsection{The modification}\label{subsec:modification}

A point estimate $t$ answers the question ``which value of
$g(\theta)$?'' with a number. The modified object answers a different
question---``how consistent is each $\theta$ with the observed
$s$?''---with a function.

\begin{definition}\label{def:generalized}
A \emph{generalized estimator} is a function
$\tau:S\times\Theta\rightarrow\mathbb{R}$ such that
\begin{enumerate}
\item[(i)] $\tau(\cdot,\theta)\in V_{\theta}$ and
$E\bigl(\tau(s,\theta)\bigr)=0$ for every
$\theta\in\Theta$;
\item[(ii)] $\tau(s,\cdot)$ is continuously differentiable on $\Theta$
for almost every $s$;
\item[(iii)] $\mathrm{Var}(\tau):=E\bigl(\tau^{2}(s,\theta)\bigr)>0$
for every $\theta$.
\end{enumerate}
For each $s$, the function $\tau_{s}=\tau(s,\cdot):\Theta\rightarrow\mathbb{R}$
is the \emph{generalized estimate} at $s$.
\end{definition}

\begin{remark}\label{rem:notation}
In conditions (i) and (iii) the expectation and variance carry no
subscript, and none is needed: the second argument of $\tau$ names
the distribution, so the expectation at $\theta$ is always computed
under $P_{\theta}$.
A point estimator $t$ is a real-valued random variable, and its
expectation and variance are real numbers; likewise, a generalized
estimator $\tau$ is a random variable taking values in the space of
smooth functions, so that $E(\tau)$ and $\mathrm{Var}(\tau)$ are
themselves smooth functions on $\Theta$. We therefore suppress the
subscript whenever
generalized estimators are discussed, and retain Bahadur's
$E_{\theta}$, $\mathrm{Var}_{\theta}$ when presenting his
development, where the estimate $t$ carries no argument that could
name the distribution. Note also that variance is written
$\mathrm{Var}$ throughout: $V_{\theta}$ is reserved for the Hilbert
space.
\end{remark}

\begin{remark}\label{rem:geometry}
The definition is geometric. Condition (i) says that at each
$\theta$ the vector $\tau(\cdot,\theta)$ is orthogonal to the
constant functions in $V_{\theta}$:
$E\bigl(\tau(s,\theta)\bigr)=(\tau(\cdot,\theta),1)_{\theta}=0$.
Condition (iii) says its length
$\|\tau\|=\sqrt{\mathrm{Var}(\tau)}$ is positive at every
distribution in the family, so the direction $\tau/\|\tau\|$ always
exists. Condition
(ii) is what allows us to ask, in Section~\ref{sec:score}, how this
direction changes as $\theta$ moves.
The two views recur throughout the paper: fixing $\theta$ makes
$\tau$ a vector in a Hilbert space; fixing $s$
makes it a smooth function on $\Theta$.
\end{remark}

Condition (i) is Bahadur's condition~\eqref{eq:membership} together
with centering: at each $\theta$, $\tau(\cdot,\theta)$
lies in $V_{\theta}$. The centering
is what makes $\tau_{s}$ readable as a continuum of test statistics:
$\tau_{s}(\theta)$ is the observed value at $s$ of a statistic that
has mean zero when $\theta$ is true, so its magnitude measures the
discrepancy between $s$ and $\theta$, simultaneously for every
$\theta$. Estimation and testing, separated at the start of the
classical theory, are here the same act. Condition (ii) gives the
estimate enough smoothness to be assessed by how it responds as
$\theta$ moves; condition (iii) rules out the degenerate case.

Every non-degenerate smooth point estimator induces a generalized
estimator. If $t$ is an
estimate in the sense of Definition~\ref{def:estimate} having
positive variance and whose mean
function $\theta\mapsto E_{\theta}(t)$ is continuously differentiable,
then
\begin{equation}
t(s)-E_{\theta}(t)\label{eq:expectation-map}
\end{equation}
satisfies (i)--(iii); we call \eqref{eq:expectation-map} the
\emph{expectation generalized estimator} associated with $t$ and use
the notation $\tau_{E(t)}$ or simply $\tau_{E}$ when considering only
one point estimator. An
unbiased estimate of $g$ induces $\tau_{E}=t-g(\theta)$, whose sign and
magnitude at $(s,\theta)$ report how far the observed estimate sits
from the value it would be estimating if $\theta$ were true. A
different generalized estimator is obtained by using the
distribution of $t$ rather than just its mean. This generalized
estimator requires the material of Lecture 13 and is introduced in
Section~\ref{sec:score}. The archetype of a generalized estimator,
the score itself, likewise appears there; nothing in
the present section depends on it.

Two payoffs can be stated immediately; the deeper ones---assessment by
information, uniform optimality, sufficiency as the equality
case---occupy Sections \ref{sec:score}--\ref{sec:heuristics}.

First, existence. A generalized estimate is a function on $\Theta$ and
does not require a well-behaved maximizer or root to exist. In
binomial sampling with $n$ trials, $y$ successes, and
$\Theta=(0,1)$, the sample $y=0$ supports no maximum likelihood
estimate of the odds $g(p)=p/(1-p)$---indeed no ML estimate of $p$ in
the open parameter space---yet the score at $y=0$,
$\tau_{y=0}(p)=-n/(1-p)$,
is defined, smooth, and usable for inference on all of $(0,1)$.

Second, invariance. The \emph{standardization}
$\bar{\tau}=\tau/\sqrt{\mathrm{Var}(\tau)}$ is the same function on the
family however the family is labeled: relabeling $\Theta$ by a smooth
bijection carries $\bar{\tau}$ to the corresponding function in the
new labels.
Geometrically, at each $\theta$,
$\bar{\tau}=\tau/\|\tau\|$ is the direction vector of
the estimator (Remark~\ref{rem:geometry}), and the assessment of
Section~\ref{sec:score} measures how fast this direction changes
with $\theta$.
Assessment built on $\bar{\tau}$ will therefore be
free of the parameterization, in contrast to mean square error,
whose value changes with the labeling of what is being estimated.

The modification discards nothing. Point estimates remain available
whenever they exist---a root of $\tau_{s}$ is a point estimate, and
\eqref{eq:expectation-map} embeds the classical objects into the generalized
class---and every theorem of Bahadur's development remains a theorem.
What changes is which object the theory is about, and, in consequence,
which assessments are natural.

\subsection{Lecture 5: risk and the approaches to making it
small}\label{subsec:risk}

Returning to Lecture 5: the quality of an estimate $t$ is classically
described through a loss function $L(t,g)$, with expected loss
\[
R_{t}(\theta)=E_{\theta}\bigl(L(t(s),g(\theta))\bigr)
\]
the \emph{risk function} of $t$; for $t$ to be good, $R_{t}$ should be
small. Bahadur motivates the square error choice by a smoothness
heuristic: if $L\ge0$ vanishes on the diagonal and is smooth in $t$,
then Taylor expansion about $t=g$ gives
$L(t,g)=\tfrac{1}{2}a(g)(t-g)^{2}+\cdots$ with $a(g)\ge0$, so that for
$a(g)>0$ the risk is locally proportional to
$E_{\theta}(t-g)^{2}$, the mean square error of $t$ at $\theta$. He
then assumes $R_{t}(\theta)=E_{\theta}(t-g)^{2}$ henceforth, records
the decomposition
\[
R_{t}(\theta)=\mathrm{Var}_{\theta}(t)+\bigl[b_{t}(\theta)\bigr]^{2},
\qquad b_{t}(\theta)=E_{\theta}(t)-g(\theta),
\]
and adds a note that will matter below: since
$EZ=\int_{0}^{\infty}P(Z\ge z)\,dz$ for $Z\ge0$,
\begin{equation}
R_{t}(\theta)=\int_{0}^{\infty}
P_{\theta}\bigl(|t(s)-g(\theta)|>\sqrt{z}\bigr)\,dz.\label{eq:pointwise-vivid}
\end{equation}
The risk at $\theta$ is a functional of the sampling distribution of
$t$ under $P_{\theta}$ and of nothing else. Equation
\eqref{eq:pointwise-vivid} makes this vivid, and it is the property on
which the lemma of Section~\ref{subsec:impossibility} turns.

There are, Bahadur says, several approaches to making $R_{t}$ small,
and he lists three before adding a fourth.
\begin{description}
\item[Admissibility:] $t$ is inadmissible if some $t'$ has
$R_{t'}(\theta)\le R_{t}(\theta)$ everywhere with strict inequality
somewhere; admissible otherwise (``the sure-thing principle'').
\item[Minimaxity:] $t_{0}$ is minimax if
$\sup_{\theta}R_{t_{0}}(\theta)\le\sup_{\theta}R_{t}(\theta)$ for all
$t$.
\item[Bayes estimation:] for a probability $\lambda$ on $\Theta$,
$t^{*}$ is Bayes if it minimizes the average risk
$\bar{R}_{t}=\int_{\Theta}R_{t}(\theta)\,d\lambda$.
\end{description}
The two one-line facts Bahadur records---a constant-risk Bayes estimate is
minimax, and an essentially unique Bayes estimate is
admissible---already display the structure of the subject: the three
notions support one another, and none of them asserts that any
estimate is uniformly best
\citep[for the standard development see][Chs.~4--5]{LehmannCasella1998}.
Bahadur develops Bayes estimation in
Lectures 6--10; we do not, for a reason the next subsection makes
precise.

The fourth approach is \emph{unbiasedness}: require
$E_{\theta}(t)=g(\theta)$, i.e.\ $b_{t}\equiv0$, so that
$R_{t}(\theta)=\mathrm{Var}_{\theta}(t)$, and then ask
(Bahadur's questions, verbatim in substance):
\begin{enumerate}
\item[(i)] Are there any unbiased estimates at all?
\item[(ii)] If so, which $t$, if any, has minimum variance at a given
$\theta$ (a \emph{locally} minimum-variance unbiased estimate)?
\item[(iii)] If a locally minimum-variance unbiased estimate exists,
is it independent of $\theta$? (If so it is the uniformly
minimum-variance unbiased estimate. If this estimate exists, what is
it?)
\end{enumerate}
Questions (i)--(iii) organize Chapters 3--5 of the lectures, and their
answers are among the permanent achievements of the classical theory.
But the questions themselves call for explanation. Why must the class
of estimates be cut down before ``best'' can be asked? Why is the
central question (iii) about the \emph{removal of a
$\theta$-dependence}? The following lemma is the explanation.

\subsection{An impossibility lemma}\label{subsec:impossibility}

Call an assessment criterion \emph{pointwise} if it assigns to an
estimate $t$ and a parameter value $\theta$ a risk
$R_{t}(\theta)\ge0$ computed from the sampling distribution of $t$
under $P_{\theta}$ alone, with $R_{t}(\theta)=0$ if and only if
$t=g(\theta)$ almost surely under $P_{\theta}$. Squared error is
pointwise---equation \eqref{eq:pointwise-vivid} exhibits
$R_{t}(\theta)$ as a functional of that sampling distribution---and so
is the risk built from any loss $L\ge0$ vanishing exactly on the
diagonal: absolute error, standardized quadratic loss, and
divergence-based losses among them.

\begin{lemma}\label{lem:pointwise}
Suppose $g$ is not constant on $\Theta$ and the family
$\{P_{\theta}\}$ has common support. If $R$ is pointwise, then no
estimate is uniformly $R$-optimal: there is no $t^{*}$ with
$R_{t^{*}}(\theta)\le R_{t}(\theta)$ for every estimate $t$ and every
$\theta\in\Theta$.
\end{lemma}

\begin{proof}
Choose $\theta_{0},\theta_{1}$ with
$g(\theta_{0})\neq g(\theta_{1})$. The constant estimate
$t\equiv g(\theta_{0})$ has $R_{t}(\theta_{0})=0$, so a uniform
optimizer $t^{*}$ must satisfy $R_{t^{*}}(\theta_{0})=0$, hence
$t^{*}=g(\theta_{0})$ a.s.\ $P_{\theta_{0}}$. The same argument at
$\theta_{1}$ gives $t^{*}=g(\theta_{1})$ a.s.\ $P_{\theta_{1}}$. By
common support, $P_{\theta_{0}}$ and $P_{\theta_{1}}$ are mutually
absolutely continuous, so $t^{*}$ is almost surely equal to two
different constants on the same sample space---a contradiction.
\end{proof}

The lemma is elementary, and that is its point. The absence of a
uniformly best estimate is not a defect of squared error that a better
loss might repair: no reweighting, no standardization, no
invariant loss escapes, because all remain pointwise. Nor is
smoothness the issue; the lemma applies equally to the
Chapman--Robbins setting of Lecture 15 where the score does not exist.
The impossibility is structural. A criterion computed distribution by
distribution cannot single out one estimate as best for the family as
a whole, because at each single distribution the criterion is
minimized by an answer---the constant---that uses no data at all.

Read against the lemma, Lecture 5's list is no longer a list of four
independent research programs. Each entry is a response to the same
impossibility, and each responds by changing the question.
Admissibility lowers the bar from ``best'' to ``not everywhere
beaten.'' Minimaxity and Bayes averaging replace the risk
\emph{function} by a scalar---its supremum or its
$\lambda$-average---so that a total order exists by construction.
Unbiasedness excludes the lemma's witness: the constant estimates are
the first casualties of the constraint $b_{t}\equiv0$, after which
minimum variance can be pursued within what remains, and questions
(i)--(iii) are exactly the questions of how far that pursuit can be
carried. These responses produced deep and useful mathematics. The
point is only that all four are responses \emph{to} the
impossibility: what forces them is the pointwise criterion, not the
estimation problem itself.

One further reading of the lemma sets up the rest of the paper. The
proof turns on the fact that a pointwise criterion evaluates $t$ at
$\theta$ using $P_{\theta}$ alone: the criterion cannot see how $t$
behaves \emph{across} the family, so it cannot penalize the constant
estimate for its uselessness everywhere else. A criterion that escapes
the lemma must therefore be \emph{family-aware}: it must use, at
$\theta$, information about the behavior of the estimator in a
neighborhood of $\theta$ within the family. Bahadur's definition of an
estimate---membership in every $V_{\theta}$ at once---already provides
the domain on which such a criterion can operate, and the generalized
estimator of Definition~\ref{def:generalized} is the object fitted to
that domain: $\tau$ has, at each $\theta$, both a distribution
(condition (i)) and a response to movement of $\theta$ (condition
(ii)).
The information $\Lambda(\tau)$ of
Section~\ref{sec:score} is built from precisely these two ingredients,
and it will do what the lemma says no pointwise criterion can: admit a
uniform optimum, with a short proof, and with no auxiliary apparatus.

\section{Unbiased estimation in Bahadur's Hilbert space}\label{sec:unbiased}

\subsection{Lecture 11: likelihood ratios and the family seen from a
single distribution}\label{subsec:LR}

The machinery of Chapter 3 of the lectures is fixed in three steps.
Fix $\theta\in\Theta$ and thereby a distribution in the family. By
mutual absolute continuity there is for each
$\delta\in\Theta$ a likelihood ratio
$\Omega_{\delta,\theta}=dP_{\delta}/dP_{\theta}$ (explicitly,
$\Omega_{\delta,\theta}=\ell_{\delta}/\ell_{\theta}$ when densities with
respect to a common $\mu$ exist), and it is assumed throughout that
$\Omega_{\delta,\theta}\in V_{\theta}$. Let
\[
W_{\theta}=\text{the subspace of }V_{\theta}\text{ spanned by }
\{\Omega_{\delta,\theta}:\delta\in\Theta\},
\]
which contains the constants since
$\Omega_{\theta,\theta}=1$. The reason $W_{\theta}$ is the right
space to address unbiasedness is a one-line identity: for any
estimate $t$ and any
$\delta$,
\begin{equation}
E_{\delta}(t)=\int_{S}t\,\Omega_{\delta,\theta}\,dP_{\theta}
=(t,\Omega_{\delta,\theta})_{\theta},\label{eq:family-inner}
\end{equation}
where $(\cdot,\cdot)_{\theta}$ is the inner product of $V_{\theta}$.
The mean of $t$ under \emph{every} member of the family is an inner
product in the \emph{single} space $V_{\theta}$: the family, as it
bears on estimation, is encoded in the geometry at $\theta$, and
$W_{\theta}$ is the family as seen from $\theta$.

\subsection{Lectures 11--12: the locally best unbiased
estimate}\label{subsec:LMVUE}

Let $U_{g}$ be the class of unbiased estimates of $g$, assumed
non-empty, so that $R_{t}(\theta)=\mathrm{Var}_{\theta}(t)$ on
$U_{g}$. Bahadur's item 8 answers question (ii) completely. Write
$\pi=\pi_{W_{\theta}}$ for orthogonal projection onto $W_{\theta}$.

\begin{bitem}{Item 8}
For $t\in U_{g}$, let $\tilde{t}=\pi t$. Then
\begin{enumerate}
\item[(a)] $\tilde{t}$ is the (essentially) unique member of
$U_{g}\cap W_{\theta}$;
\item[(b)] $\tilde{t}$ is the projection onto $W_{\theta}$ of
\emph{every} $t\in U_{g}$;
\item[(c)] $\tilde{t}$ has minimum variance at $\theta$ within
$U_{g}$.
\end{enumerate}
\end{bitem}

The proof is the
geometry of Section~\ref{subsec:formulation} plus
identity~\eqref{eq:family-inner}: since $\pi$ is self-adjoint and
fixes $W_{\theta}$,
\[
E_{\delta}(\pi t)=(\pi t,\Omega_{\delta,\theta})_{\theta}
=(t,\pi\Omega_{\delta,\theta})_{\theta}
=(t,\Omega_{\delta,\theta})_{\theta}=E_{\delta}(t)
\quad\text{for all }\delta\in\Theta,
\]
so projection onto $W_{\theta}$ preserves every expectation in the
family---it discards exactly the part of $t$ the family cannot
see---and, being a projection, it does not increase the norm, hence
not the variance. The locally minimum-variance unbiased estimate
(LMVUE) at $\theta$ is the family-visible part of any unbiased
estimate.

Lecture 12 restates and extends this.

\begin{bitem}{Item 9}
An estimate $t$ with finite variances is the LMVUE at $\theta$ of
its own expected value if and only if $t\in W_{\theta}$; it is the
UMVUE if and only if
$t\in C=\bigcap_{\theta\in\Theta}W_{\theta}$.
\end{bitem}

\begin{bitem}{Item 10 (Lehmann--Scheff\'{e})}
Write $\tilde{U}$ for the class of $u$ with $E_{\delta}(u)=0$ for
all $\delta$. A $t$ with finite variances belongs to $C$ if and only
if $E_{\theta}(tu)=0$ for every $\theta\in\Theta$ and every
$u\in\tilde{U}$.
\end{bitem}

Item 10 is due to \citet{LehmannScheffe1950}. Questions
(i)--(iii) are thereby reduced to concrete geometry: existence of
unbiased estimates, projection at each $\theta$, and the intersection
$C$. To item 8 Bahadur appends a note, and the note is the hinge of this
paper: ``$\tilde{t}$ often depends on $\theta$, and this is the
problem in practice.''

\subsection{The problem in practice}\label{subsec:M3}

The classical response to the note is question (iii): find conditions
under which the $\theta$-dependence vanishes. Chapters 4 and 5 of the
lectures can be read as the systematic pursuit of this
question---spans of likelihood derivatives, attainment of bounds,
exponential structure, sufficiency---and the pursuit succeeds exactly
in the families where $W_{\theta}$ essentially does not depend on
$\theta$. Outside those families the LMVUE remains locally tuned, and
the classical theory regards this as its central practical defect.

The modification of Section~\ref{subsec:modification} responds
differently: it assembles what the classical theory already built.
The map $\theta\mapsto\tilde{t}$, with
$\tilde{t}=\pi_{W_{\theta}}t$ the LMVUE at $\theta$, defines
a single function on $S\times\Theta$, namely
$(s,\theta)\mapsto\tilde{t}(s)$. Recentering this yields
\begin{equation}
\tau(s,\theta)=\tilde{t}(s)-g(\theta).\label{eq:assembled}
\end{equation}
Since $\tilde{t}\in U_{g}$, we have
$E\bigl(\tau(s,\theta)\bigr)=0$ for every
$\theta$---condition (i) of Definition~\ref{def:generalized}
exactly---and, granted the smoothness in $\theta$ that Bahadur's next
chapter assumes anyway, $\tau$ is a generalized estimator. The
classical theory's own locally best objects, granted their
$\theta$-dependence rather than pardoned for it, assemble into one
generalized estimator. Question (iii) asks when \eqref{eq:assembled}
collapses to the point-estimator form $t-g(\theta)$, constant in
$\theta$ up to the recentering. The question for generalized
estimation is not
whether the object collapses but how good is it, and
Section~\ref{sec:score} answers: its first-order version---the
projection onto Bahadur's span of $\{1,\gamma_{\theta}^{(1)}\}$,
recentered---is $g'(\theta)\gamma_{\theta}^{(1)}/I$, which has the direction of
the score and is fully efficient under the information criterion
introduced there.

Seen this way, the local dependence is not a pathology. At each
$\theta$, $\tilde{t}$ is tuned to distinguish $P_{\theta}$
from its neighbors in the family; a single $t$ that is best
everywhere exists only in the special families where one function can
be simultaneously tuned to every neighborhood. Bahadur's phrase
describes the structure of the correct inferential object, encumbered by
a theory committed to a different one.

\subsection{Lecture 12: a two-point parameter
space}\label{subsec:two-point}

Lecture 12 closes with Example 1(d), which deserves retelling because
Bahadur draws the moral himself. Let $s=(X_{1},\ldots,X_{n})$ with
$X_{i}$ iid $N(\theta,1)$ and $\Theta=\{1,2\}$, and let
$g(\theta)=\theta$. Here
$W_{\theta}=\mathrm{Span}\{1,e^{n\bar{X}}\}$, so the LMVUE of $g$ at
$\theta=1$ has the form $a+be^{n\bar{X}}$ with $a,b$ determined by
the two unbiasedness equations. The result is an ``estimate'' of a
parameter known to lie in $\{1,2\}$ that ranges over
$(-\infty,\infty)$, with $b>0$, and the LMVUE at $\theta=2$ is a
different function altogether; $C$ contains only constants, so no
non-trivial UMVUE exists. Bahadur: ``This is absurd. MSE is not
suitable because $g$ takes on only two values.'' His prescriptions:
the Neyman--Pearson theory
\citep{NeymanPearson1933} indicates an estimate of the form
$a+bI_{A}$ with $A=\{\bar{X}>c\}$; and ``we should restrict the
estimation theory to a continuum of values (i.e., should have only
connected $\Theta$)''---the standing assumption of
Section~\ref{subsec:formulation} and very close to Fisher's view of
estimation as a continuum of significance tests.

The generalized reading agrees with both prescriptions and explains
them. With two distributions the inferential task is transparently
discrimination, and the natural object is a test statistic---which is
what a generalized estimate is at every $\theta$, here reduced to a
pair. No absurdity can arise because nothing is required to take
values in $\Theta$. At the same time, the generalized theory does not
manufacture an estimator where Bahadur says none should be sought:
condition (ii) of Definition~\ref{def:generalized} requires a
continuum of parameter values just as Bahadur's prescription does,
because the assessment of Section~\ref{sec:score} is built from
derivatives in $\theta$. On a finite $\Theta$ both theories hand the
problem to Neyman and Pearson, the classical one by retreat, the
generalized one by construction.

\section{Score, information, and bounds}\label{sec:score}

\subsection{Lecture 13: the score and the spans
\texorpdfstring{$W_{\theta}^{(k)}$}{W(k)}\label{subsec:L13}}

Let $\Theta$ be an open interval, $dP_{\theta}=\ell_{\theta}\,d\mu$
with $\ell_{\theta}(s)>0$, and let dashes denote derivatives with
respect to the parameter, as in the lectures. The \emph{score} at
$\theta$ is
\[
\gamma_{\theta}^{(1)}(s)=\frac{\ell_{\theta}'(s)}{\ell_{\theta}(s)}
=\Omega_{\delta,\theta}'\big|_{\delta=\theta},
\]
and more generally
$\gamma_{\theta}^{(j)}=\ell_{\theta}^{(j)}/\ell_{\theta}$, with
$E_{\theta}(\gamma_{\theta}^{(j)})=0$ for all $j$ under the differentiation
conditions Bahadur defers (as do we) to the exact statements of
Lecture 16. The \emph{Fisher information} is
$I=E_{\theta}\bigl(\gamma_{\theta}^{(1)}\bigr)^{2}
=\mathrm{Var}_{\theta}(\gamma_{\theta}^{(1)})$, and it is additive across
independent experiments.

Much of the chapter consists of instances of a single identity, and
it is worth stating the identity. For suitably regular
$f:S\times\Theta\rightarrow\mathbb{R}$ with
$f(\cdot,\theta)\in V_{\theta}$, differentiating
$E_{\theta}\bigl(f(s,\theta)\bigr)
=\int f(s,\theta)\,\ell_{\theta}(s)\,d\mu$ in $\theta$ gives
\begin{equation}
\bigl(E_{\theta}f\bigr)'
=E_{\theta}(f')+E_{\theta}\bigl(\gamma_{\theta}^{(1)}f\bigr).\label{eq:master}
\end{equation}
For $f=t$ free of $\theta$ and unbiased for $g$,
\eqref{eq:master} reads
$g'(\theta)=E_{\theta}(\gamma_{\theta}^{(1)}t)
=(\gamma_{\theta}^{(1)},t)_{\theta}$---Bahadur's key computation in this
lecture. For $f=\gamma_{\theta}^{(1)}$ it reads
$0=E_{\theta}(L_{\theta}'')+I$, since
$\gamma_{\theta}^{(1)}=L_{\theta}'$ for $L_{\theta}=\log\ell_{\theta}$; this
is Bahadur's item 13 (displayed below, after item 12),
$E_{\theta}(L_{\theta}'')=-I$, and iterating
\eqref{eq:master} on the score produces the Bartlett identities
\citep{Bartlett1953a}. The
identity will carry the main theorem of
Section~\ref{subsec:Lambda} as well; the lectures and the
generalized theory run on the same engine.

The Taylor expansion
\[
\Omega_{\delta,\theta}
=1+(\delta-\theta)\gamma_{\theta}^{(1)}
+\tfrac{1}{2}(\delta-\theta)^{2}\gamma_{\theta}^{(2)}+\cdots
\]
 suggests
$W_{\theta}=\mathrm{Span}\{1,\gamma_{\theta}^{(1)},\gamma_{\theta}^{(2)},\ldots\}$,
an equality that holds exactly in a one-parameter exponential family
(Lecture 17) and approximately in large samples. In any case the
finite spans
$W_{\theta}^{(k)}
=\mathrm{Span}\{1,\gamma_{\theta}^{(1)},\ldots,\gamma_{\theta}^{(k)}\}
\subseteq W_{\theta}$ are available, every $t\in U_{g}$ has the same
projection $t_{\theta,k}^{*}$ onto $W_{\theta}^{(k)}$, and comparing
norms gives:

\begin{bitem}{Item 11 (Bhattacharya bounds)}
For each $t\in U_{g}$ and $k=1,2,\ldots$,
\[
\mathrm{Var}_{\theta}(t)\ge
E_{\theta}(t_{\theta,k}^{*})^{2}-g(\theta)^{2},
\]
and the bounds are non-decreasing in $k$.
\end{bitem}

For $k=1$,
$\{1,\gamma_{\theta}^{(1)}/\sqrt{I}\}$ is an orthonormal basis of
$W_{\theta}^{(1)}$; using $(1,t)_{\theta}=g(\theta)$ and
$(\gamma_{\theta}^{(1)},t)_{\theta}=g'(\theta)$,
\begin{equation}
t_{\theta,1}^{*}
=g(\theta)+\frac{g'(\theta)}{I}\,\gamma_{\theta}^{(1)},\qquad
\|t_{\theta,1}^{*}\|^{2}=g(\theta)^{2}+\frac{(g'(\theta))^{2}}{I},
\label{eq:k1-projection}
\end{equation}
and the $k=1$ bound is the \emph{information inequality}:

\begin{bitem}{Item 12 (Fisher--Darmois--Cram\'{e}r--Rao)}
For all $t\in U_{g}$,
\begin{equation}
\mathrm{Var}_{\theta}(t)\ge\frac{(g'(\theta))^{2}}{I}.\label{eq:CR}
\end{equation}
\end{bitem}

The bounds of item 11 are due to \citet{Bhattacharyya1946}; the
inequality of item 12 was obtained independently by \citet{Rao1945},
\citet{Darmois1945}, and \citet{Cramer1946}.

\begin{bitem}{Item 13}
With $L_{\theta}=\log\ell_{\theta}$,
\[
E_{\theta}\bigl(L_{\theta}''(s)\bigr)=-I(\theta).
\]
\end{bitem}

\subsection{The information of a generalized estimator}\label{subsec:Lambda}

The convention of Remark~\ref{rem:notation} is in force: $E$ and
$\mathrm{Var}$ applied to generalized estimators carry no subscript.
The quantity now to be introduced, $\Lambda$, joins them: it is a
function on $\Theta$ attached to the family, and under a smooth
relabeling of $\Theta$, $\Lambda$ and $I$ transform identically, so
that ratios of the two are label-free. This is the invariance
promised in Section~\ref{subsec:modification}.

Let $\tau$ be a generalized estimator
(Definition~\ref{def:generalized}), $\mathrm{Var}(\tau)$
its variance function, and
$\bar{\tau}=\tau/\sqrt{\mathrm{Var}(\tau)}$ its standardization.

\begin{definition}\label{def:Lambda}
The \emph{information utilized by} $\tau$ is
\begin{equation}
\Lambda(\tau)=\bigl(E(\bar{\tau}')\bigr)^{2}
=\frac{\bigl(E(\tau')\bigr)^{2}}{\mathrm{Var}(\tau)}.\label{eq:Lambda}
\end{equation}
\end{definition}

The definition already fixes the statistical content of $\Lambda$, and
it is worth reading off before any bound. A generalized estimator uses
the sample $s$ to discriminate among the distributions of the family,
and it does this well when its estimate $\tau_{s}(\theta)$ changes
quickly as $\theta$ moves---when the graph of $\tau_{s}$ over $\Theta$
is steep. No estimator has the steepest graph at every $s$ at once, but
the slope averaged over $s$ has a largest value, and it is this average
that $\Lambda$ records: with the standardization $\bar{\tau}$ that puts
all estimators on equal footing,
$\Lambda^{1/2}(\tau)=\bigl|E(\bar{\tau}')\bigr|$ is the average slope of
$\bar{\tau}$. The reading is at once geometric and statistical.
Geometrically, $\bar{\tau}$ is a direction---a unit vector orthogonal to
the constants at every $\theta$---and $\Lambda^{1/2}$ is the mean rate
at which that direction turns as $\theta$ moves. Statistically,
$\bar{\tau}$ is a test statistic with mean zero and standard deviation
one at every $\theta$, and $\Lambda^{1/2}$ is the mean rate at which it
responds to the hypothesis being tested. That one quantity carries both
readings---the fundamental object of the geometry and the
standardization basic to statistical comparison---is part of what
recommends it. $\Lambda$ is thus family-aware in exactly the way
Lemma~\ref{lem:pointwise} requires, and at an isolated distribution it
is undefined. The lemma's witness fares as badly as it should: a
constant estimate $t\equiv c$ induces $\tau_{E(t)}=c-E(c)=0$, excluded
by condition (iii), and any estimator whose mean response $E(\tau')$
vanishes identically has $\Lambda=0$---the worst value, not the best.
The criterion punishes exactly what pointwise criteria reward.

Applying \eqref{eq:master} to a generalized estimator, for which
$E\bigl(\tau(s,\theta)\bigr)\equiv0$ on $\Theta$, gives
the \emph{score equation}
\begin{equation}
E(\tau')+E\bigl(\gamma^{(1)}\tau\bigr)=0.
\label{eq:score-eq}
\end{equation}

\begin{theorem}\label{thm:optimality}
For every generalized estimator $\tau$ and every
$\theta\in\Theta$,
\[
\Lambda(\tau)
=\Bigl[\mathrm{Corr}\bigl(\tau,\gamma^{(1)}\bigr)\Bigr]^{2}\,I
\;\le\;I,
\]
with equality at $\theta$ if and only if $\tau(\cdot,\theta)$ is
almost surely a scalar multiple of $\gamma_{\theta}^{(1)}$. The score is a
generalized estimator with $\Lambda(\gamma^{(1)})=I$; it attains the
bound at every $\theta\in\Theta$.
\end{theorem}

\begin{proof}
By \eqref{eq:score-eq},
$E(\bar{\tau}')=-E(\gamma^{(1)}\bar{\tau})$, so
$\Lambda(\tau)=E^{2}(\gamma^{(1)}\bar{\tau})
=\mathrm{Corr}^{2}(\tau,\gamma^{(1)})\,I$, and the
inequality with its equality condition is Cauchy--Schwarz. Taking
$\tau=\gamma^{(1)}$ gives correlation one.
\end{proof}

Efficiency is classically a matter of variances---the ratio of the
Cram\'{e}r--Rao bound to the variance of an estimator, or of one
variance to another. Here it is a matter of information, and, by the
reading just given, information is squared average slope.

\begin{definition}\label{def:Eff}
The \emph{efficiency} of a generalized estimator $\tau$ is
\[
\mathrm{Eff}(\tau)=\frac{\Lambda(\tau)}{I},
\]
the information utilized by $\tau$ as a fraction of the Fisher
information $I$.
\end{definition}

The definition carries the statistical content; the theorem supplies the geometry by which we obtain
the bound and its extremal case.

\begin{corollary}\label{cor:eff-corr}
For every generalized estimator $\tau$ and every $\theta\in\Theta$,
\[
\mathrm{Eff}(\tau)=\mathrm{Corr}^{2}\bigl(\tau,\gamma^{(1)}\bigr)\in[0,1],
\]
with $\mathrm{Eff}(\tau)=1$ if and only if $\tau(\cdot,\theta)$ is
almost surely a scalar multiple of $\gamma_{\theta}^{(1)}$. The score is
the unique direction of full efficiency.
\end{corollary}

Three comments on the theorem and its corollary, in increasing order of
importance.

First, the optimum is a direction, not a single function: any
$A(\theta)\gamma^{(1)}$ with $A$ smooth and non-vanishing attains the
bound. The multiplicity is no qualification. At each $\theta$ a
generalized estimate is a test statistic, and rescaling by
$A(\theta)>0$ changes no rejection region; the geometry identifies
exactly what the statistics cannot distinguish.
Section~\ref{sec:discussion} returns to this point.

Second, the theorem should be taken at full strength. It is uniform
over $\Theta$; its proof is three lines from
identity~\eqref{eq:master}, the same identity the lectures already
run on; and it requires no restriction of the competing class---no
unbiasedness, no exclusion of degenerate rivals, no auxiliary notion
to break ties.
The competing class is in fact larger than Bahadur's: the theorem
quantifies over every $\tau$ with $\tau(\cdot,\theta)\in V_{\theta}$,
not only over $W_{\theta}$.

Third, the classical inequality \eqref{eq:CR} is contained in the
theorem, read in the other direction. For $t\in U_{g}$ the
expectation generalized estimator of
Section~\ref{subsec:modification} is $\tau_{E(t)}=t-g(\theta)$, for which
$\tau_{E(t)}'=-g'$ and $\mathrm{Var}(\tau_{E(t)})=\mathrm{Var}(t)$, so
\begin{equation}
\Lambda(\tau_{E(t)})=\frac{(g'(\theta))^{2}}{\mathrm{Var}(t)}
\;\le\;I,\label{eq:CR-rearranged}
\end{equation}
which is \eqref{eq:CR} rearranged. Bahadur reads the inequality as
bounding the variance of $t$, a point-wise criterion, by a quantity computed from
the family; \eqref{eq:CR-rearranged} reads the same
inequality as assessing the information, a family-aware criterion, against the
family benchmark---like compared with like. And the generalized
reading needs no unbiasedness: for arbitrary $t$ with smooth mean
function $h(\theta)=E_{\theta}(t)$,
$\Lambda(\tau_{E(t)})=(h')^{2}/\mathrm{Var}(t)$, the same
formula with the mean function in place of the estimand. The
constraint $b_{t}\equiv0$ was a response to pointwise assessment
(Section~\ref{subsec:impossibility}); the information criterion
assesses biased and unbiased estimators by one formula.

The assembled LMVUE of Section~\ref{subsec:M3} can now be assessed.
By \eqref{eq:k1-projection}, its first-order version recentered is
$t_{\theta,1}^{*}-g(\theta)=g'(\theta)\gamma^{(1)}/I$: a smooth
multiple of the score, hence fully efficient. The object the classical theory constructs at each
$\theta$ and then asks to stop depending on $\theta$ is, granted its
$\theta$-dependence, optimal.
Bahadur's higher spans now take their place beside the apparatus of
Section~\ref{subsec:impossibility}. Write
$\tau(\cdot,\theta)=a\gamma^{(1)}+\tau^{\perp}$ with
$\tau^{\perp}\perp\gamma^{(1)}$ at $\theta$: the numerator of
$\Lambda$ sees only the score component, since
$E(\gamma^{(1)}\tau)$ is unchanged by $\tau^{\perp}$, while
$\mathrm{Var}(\tau)$ grows with it. Components in
$W_{\theta}^{(k)}\setminus W_{\theta}^{(1)}$---indeed any component
orthogonal to the score---add variance and never information, so the
optimal generalized estimator lies in the span of the score, and the
higher spans cannot improve it. They matter classically because an
estimator constrained to be unbiased for a fixed $g$ cannot in
general lie in $W_{\theta}^{(1)}$, and the Bhattacharya bounds price
that constraint (Lecture 18's negative binomial, where
$b_{k}=\mathrm{Var}_{\theta}(\pi_{W_{\theta}^{(k)}}t)
\uparrow\sigma^{2}$, is exactly this).
Once the constraint is
dropped, the higher spans join admissibility, minimaxity, and
unbiasedness: apparatus that is not needed when a generalized
estimator is the inferential object.

The generalized estimator promised in
Section~\ref{subsec:modification}---built from the distribution of
$t$ rather than just its mean---is now available, and it is the more
principled of the two. A statistic $t$ induces the marginal family
$\{P_{\theta}\circ t^{-1}:\theta\in\Theta\}$ on its range; let the
induced family satisfy the conditions of this section. The
\emph{score generalized estimator} associated with $t$, written
$\tau_{t}$, is the score of this marginal family, composed with $t$
so as to be a function on $S\times\Theta$. 
Theorem~\ref{thm:optimality} applied within the marginal family gives
\begin{equation}
\Lambda(\tau_{t})=I_{t},\label{eq:score-map}
\end{equation}
the Fisher information of the marginal model of $t$, so that
$\mathrm{Eff}(\tau_{t})=I_{t}/I$ is the fraction of the
Fisher information retained by the statistic. This places the
criterion in the information-loss program of
\citet{Fisher1922-ie,Fisher1925},
computed for maximum likelihood in curved exponential families by
\citet{Efron1975-ic}; the equality case---$I_{t}=I$ exactly when $t$
is sufficient---is developed with Bahadur's Lecture 17 in
Section~\ref{sec:attainment}. Two properties distinguish the two
constructions. The score generalized estimator depends on $t$ only
through the $\sigma$-field
it generates: $\tau_{\psi(t)}=\tau_{t}$ for injective $\psi$,
so relabeling the estimator does not change the assessment. The
expectation version lacks this invariance---for nonlinear $\psi$,
$\tau_{E(\psi(t))}$ and $\tau_{E(t)}$ are different generalized
estimators---but remains operational when only moment functions of
$t$ are available. Among generalized estimators that are, at each
$\theta$, functions of $t$, the marginal score maximizes the
correlation with $\gamma^{(1)}$ \citep{Vos2026-mse}: $\tau_{t}$ is
the best that can be built from $t$ alone, and
$\mathrm{Eff}(\tau_{E(t)})\le\mathrm{Eff}(\tau_{t})$.

\subsection{The running example: Bernoulli sampling}\label{subsec:example2}

Bahadur's Example 2 threads through Lectures 14, 17, and 18, and it
will thread through the remainder of this paper.
Let $X_{1},X_{2},\ldots$ be iid Bernoulli$(\theta)$ with
$\Theta=(0,1)$, observed under a stopping rule: (a) stop at a fixed
$n$; (b) stop at the $k$th success (negative binomial sampling);
(c) a two-stage scheme, which we list for fidelity to the source but
do not discuss further.
In all cases the likelihood is
\[
\ell_{\theta}(s)=\theta^{T(s)}(1-\theta)^{N(s)-T(s)},
\]
with $N$ the number of trials and $T$ the number of successes, so the
score is one formula across all three plans:
\begin{equation}
\gamma_{\theta}^{(1)}(s)
=\frac{T(s)}{\theta}-\frac{V(s)}{1-\theta},
\qquad V(s)=N(s)-T(s).\label{eq:bernoulli-score}
\end{equation}
From $E_{\theta}(\gamma_{\theta}^{(1)})=0$ and item 13,
$E_{\theta}(T)/\theta=E_{\theta}(V)/(1-\theta)$ and
$I=E_{\theta}(T)/\theta^{2}+E_{\theta}(V)/(1-\theta)^{2}$. Under
plan (a), $I=n/\bigl(\theta(1-\theta)\bigr)$; under plan (b), $T=k$
and $E_{\theta}(N)=k/\theta$, whence
$I=k/\bigl(\theta^{2}(1-\theta)\bigr)$.

The example makes the family-awareness of the criterion concrete.
Under both plans the score is given by the same formula
\eqref{eq:bernoulli-score}, but the estimators are not the same
object. A generalized estimator is defined jointly on a sample space
and a family of distributions, and plans (a) and (b) generate two
different families---on two different sample spaces---that merely
share the labeling set $\Theta$. What is shared in
\eqref{eq:bernoulli-score} is notation.
The quantities that assess
the estimator---$I$, $\mathrm{Var}(\tau)$, and the standardization
$\bar{\tau}$---are expectations, and expectations are taken in the
family the sampling plan generates: under plan (a) the score utilizes
information $n/\bigl(\theta(1-\theta)\bigr)$, under plan (b)
$k/\bigl(\theta^{2}(1-\theta)\bigr)$. The assessment of an estimator
is a property of the estimator \emph{and the family}; $\Lambda$ makes
the dependence explicit.
The binomial and negative binomial comparison is developed in
\citet{Vos2026-slp}: the two likelihoods are proportional as functions
of $\theta$ and the raw scores agree, but the standardized scores
$\gamma^{(1)}/\sqrt{I}$ are different functions of $\theta$, because
the shared parameter labels different families with different Fisher
information.

The example also delivers, in its simplest form, the existence
payoff claimed in Section~\ref{subsec:modification}. Under plan (a)
with $T=0$, the score estimate is
$\tau_{s}(\theta)=-n/(1-\theta)$, smooth on all of $(0,1)$, and the
standardized estimate
$\bar{\tau}_{s}(\theta)=-\sqrt{n\theta/(1-\theta)}$ yields
inferential statements directly: the set
$\{\theta:|\bar{\tau}_{s}(\theta)|\le z\}=
\bigl(0,\,z^{2}/(n+z^{2})\bigr]$ is a one-sided interval for
$\theta$ obtained at a sample where the maximum likelihood estimate
does not exist in $\Theta$ \citep[cf.][Example 1]{VosWu2025}.
At interior samples the same set is the score interval of
\citet{Wilson1927}, and for each $\theta_{0}$ the value
$\bar{\tau}_{s}(\theta_{0})$ is the signed root of the score test
statistic of \citet{Rao1948} for the hypothesis $\theta=\theta_{0}$. The
odds, where the contrast with unbiased estimation is
starkest, waits for Lecture 18 material in
Section~\ref{sec:attainment}.

\subsection{Lecture 15: attainment and the nonsmooth
case}\label{subsec:L15}

Two results of Lecture 15, identified as items 13(a) and 14, close
the chapter and set up the next.

\begin{bitem}{Item 13(a)}
If $t\in U_{g}$ attains the bound \eqref{eq:CR} for
\emph{every} $\theta\in\Theta$, then the family is a one-parameter
exponential family,
$\ell_{\theta}(s)=\varphi(s)e^{A(\theta)+B(\theta)t(s)}$, with
$g=-A'/B'$.
\end{bitem}

The same argument as in the proof of \eqref{eq:CR} shows
$t\in W_{\theta}^{(1)}$ for all $\theta$, and the differential
equation this forces on $L_{\theta}$ integrates to the exponential
form. Uniform attainment is thus not a property a point estimator earns
by cleverness; it is a property a \emph{family} either has or lacks.
Section~\ref{sec:attainment} develops the converse direction and its
consequence, sufficiency.

Second, for families too rough for the score to exist---Bahadur's
instance is the double exponential location family---item 14
\citep{ChapmanRobbins1951} replaces derivatives with difference
quotients of likelihood ratios:

\begin{bitem}{Item 14 (Chapman--Robbins)}
For $t\in U_{g}$,
\[
\mathrm{Var}_{\theta}(t)\ge
\varlimsup_{\delta\rightarrow\theta}
\Bigl(\frac{g(\delta)-g(\theta)}{\delta-\theta}\Bigr)^{2}\Big/
E_{\theta}\Bigl(\frac{\Omega_{\delta,\theta}-1}{\delta-\theta}
\Bigr)^{2}.
\]
\end{bitem}

The bound recovers \eqref{eq:CR} when the smoothness returns. Two remarks
keep the bookkeeping honest. Chapman--Robbins is as pointwise as its
smooth counterpart, so Lemma~\ref{lem:pointwise} applies:
nonsmoothness is no escape from the impossibility. And the
generalized theory as developed here shares the smoothness
requirement---condition (ii) and the derivative in
\eqref{eq:Lambda}---so its scope matches that of Bahadur's smooth
chapters; what replaces $\Lambda$ off the smooth case is not pursued
in this paper.

\section{Attainment and the exponential family}\label{sec:attainment}

\subsection{Lecture 16: exact statements}\label{subsec:L16}

The deferred conditions are settled in Lecture 16, and we record them
in the informal register in which we will use them.
\begin{description}
\item[Condition 1:] $\ell_{\theta}(s)>0$ and
$\delta\mapsto\ell_{\delta}(s)$ is continuously differentiable for
each $s$.
\item[Condition 2:] Near each $\theta$ the scores are
dominated: $m_{\theta}(s)
=\sup_{|\delta-\theta|\le\varepsilon}|\gamma_{\delta}^{(1)}(s)|$
satisfies $E_{\theta}(m_{\theta}^{2})<\infty$ for some
$\varepsilon=\varepsilon(\theta)>0$.
\item[Condition 3:] $I(\theta)>0$.
\end{description}
Under these conditions:

\begin{bitem}{Item 12E}
If $U_{g}$ is non-empty, then $g$ is differentiable and the
information inequality \eqref{eq:CR} holds on $\Theta$.
\end{bitem}

The proof shows that the difference
quotients $(\Omega_{\delta,\theta}-1)/(\delta-\theta)$ converge in
$V_{\theta}$ to $\gamma_{\theta}^{(1)}$, so that
$\gamma_{\theta}^{(1)}\in W_{\theta}$ and identity \eqref{eq:master} is
legitimate for the $f$ we use it on. The $k$-th order versions
(conditions $1_{k}$--$3_{k}$ and statement 11E) do the same for the
Bhattacharya bounds. These conditions are what ``suitably regular''
meant in Sections \ref{subsec:modification} and \ref{subsec:Lambda},
and they are assumed from here on.

Lecture 16 also introduces the contrast case that will matter in
Section~\ref{sec:heuristics}: for $X_{i}$ iid with density
$ae^{-b(x-\theta)^{4}}$, the joint density has the exponential form
in the \emph{three} statistics
$(\sum X_{i},\sum X_{i}^{2},\sum X_{i}^{3})$ with coefficients
functions of the single parameter $\theta$---a \emph{curved}
exponential family, in Efron's later terminology, and the setting in
which the information loss of maximum likelihood is non-zero
\citep{Efron1975-ic}.

\subsection{Lecture 17: the one-parameter exponential family
(item 15)}\label{subsec:L17}

Bahadur's item 15 is the structural theorem of the lectures: it
identifies the family structure---one-parameter exponential---under
which the spaces $W_{\theta}$, the Bhattacharya bounds, and the
program of questions (i)--(iii) all become explicit and complete.

\begin{bitem}{Item 15}
Let
\[
\ell_{\theta}(s)=C(s)\,e^{A(\theta)+B(\theta)T(s)},
\]
with $C>0$, $T$ a fixed statistic, $B$ continuous and strictly
monotone, and  $\xi=B(\delta)-B(\theta)$ covering a
neighborhood of $0$. Then:
\begin{enumerate}
\item[(a)] $W_{\theta}^{(k)}=\mathrm{Span}\{1,T,\ldots,T^{k}\}$ for
$k=1,2,\ldots$;
\item[(b--c)] $W_{\theta}$ is the space of \emph{all}
square-integrable Borel functions of $T$;
\item[(d)] If $U_g$ is non-empty, the Bhattacharya bounds converge:
$b_{k}(\theta)\rightarrow\mathrm{Var}_{\theta}(\tilde{t})$;
\item[(e)] $\tilde{t}=E_{\theta}(t\mid T)$ for every $t\in U_{g}$
and every $\theta$;
\item[(f)] $T$ is \emph{sufficient}: for all $A\in\mathcal{A}$, $P_{\theta}(A\mid T)$ has a
version free of $\theta$.
\end{enumerate}
\end{bitem}

The proof of (e) is the promise of
Lecture 4 kept: once $W_{\theta}$ is exactly the square-integrable
functions of $T$, projection onto $W_{\theta}$ \emph{is} conditional
expectation given $T$, and the LMVUE machinery of
Section~\ref{sec:unbiased} becomes Blackwellization
\citep{Blackwell1947}. Note the
direction of the derivation: sufficiency is not assumed but
\emph{derived}, an output of the Hilbert-space geometry.

\subsection{Sufficiency and the one-parameter exponential
family}\label{subsec:L17gen}

Item 15 can be read as a theorem about sufficiency. In the
one-parameter exponential family the space $W_{\theta}$ is the same at
every $\theta$---the square-integrable functions of $T$, parts (b)--(c)
of item 15---so projection onto it is conditional expectation given
$T$ (part (e)), and $T$ is sufficient (part (f)).  A single statistic
$T$ carries all the information in the sample, the locally best
unbiased estimate at every $\theta$ is a function of it, and the
$\theta$-dependence that Section~\ref{subsec:M3} called ``the problem
in practice'' disappears because $W_{\theta}$ has stopped moving.
Among families with common support this is essentially the only place
it can happen: the exponential family is the one family admitting a
sufficient statistic whose dimension does not grow with the sample
size (the Pitman--Koopman--Darmois theorem;
\citealp{Pitman1936,Koopman1936}). The exponential family
enters the classical theory as the apparatus that secures
sufficiency, in the same role that admissibility, minimaxity, and
unbiasedness played in securing an optimality the pointwise criterion
could not deliver (Section~\ref{subsec:impossibility}).

Generalized estimation needs no such apparatus, because the sufficiency
content is already carried by the score. Theorem~\ref{thm:optimality},
applied inside the marginal family of a statistic $t$, gives
$\Lambda(\tau_{t})=I_{t}\le I$ (equation~\eqref{eq:score-map}), and the
equality case is exactly sufficiency: $I_{t}=I$ for every $\theta$ if
and only if $t$ is sufficient---equivalently, if and only if the score
generalized estimator of $t$ is the full-data score, $\tau_{t}=\gamma^{(1)}$,
so that $t$ loses none of the information the score utilizes.
Where the classical theory must find a family rigid enough to hold a
fixed sufficient statistic---$W_{\theta}$ frozen across $\theta$, one
$T$ good at every distribution at once---the generalized theory asks
only for a direction, the score $\gamma_{\theta}^{(1)}$, which is
defined separately at each $\theta$ and need not be the same function
of the data as $\theta$ varies. A fixed sufficient statistic is a
demanding, global requirement on the family, met by few; the score is
a local object that exists under smoothness alone (Lecture 16,
Section~\ref{subsec:L16}) and so is available in every family the
theory treats. 
Fisher's program---sufficient statistics as lossless reductions,
efficiency as the fraction of information retained
\citep{Fisher1922-ie}---is captured by the correlation between
$\tau_{t}$ and the score:
$\mathrm{Eff}(\tau_{t})=\Lambda(\tau_{t})/I=\mathrm{Corr}(\tau_{t},\gamma^{(1)})^{2}$.

What generalized estimation says about the exponential family is
nonetheless fully consistent with the classical account, and it
returns the attainment theorem almost for free. Here the score is
\begin{equation}
\gamma^{(1)}=A'(\theta)+B'(\theta)\,T
=B'(\theta)\bigl(T-\mu(\theta)\bigr),
\qquad\mu(\theta)=E_{\theta}(T)=-\frac{A'(\theta)}{B'(\theta)},
\label{eq:expfam-score}
\end{equation}
a smooth multiple of the centered sufficient statistic, so the
expectation generalized estimator $\tau_{E(T)}=T-\mu(\theta)$ has the
score's direction at every $\theta$: $\mathrm{Eff}(\tau_{E(T)})=1$
identically, and $T$ attains the Cram\'{e}r--Rao bound uniformly. Read
together with item 13(a) of Section~\ref{subsec:L15}, this is the
complete account of uniform attainment. Item 13(a) says a point
estimator can meet the bound at every $\theta$ only in a family of the
form \eqref{eq:expfam-score}; the generalized reading supplies the
reason---attainment at $\theta$ is alignment of $\tau_{E(t)}$ with the
score direction there (Theorem~\ref{thm:optimality}), and simultaneous
alignment at every $\theta$ is possible only when the family supplies a
single statistic whose centered version already points along the score
everywhere, which is precisely the exponential form. 

The exponential family is the flat case of this picture: the score is a
multiple of one sufficient statistic, its direction is fixed by that
statistic, and attainment is exact. The curved case---where no fixed
statistic carries the score and the maximum likelihood estimator falls
short of full efficiency---is Efron's subject, and the shortfall sits
inside Bahadur's geometry. Efron's statistical curvature $\kappa_{\theta}$
is a property of the family alone: the normalized length of the part of
the second log-likelihood derivative $L_{\theta}''$ that the score span
$W_{\theta}^{(1)}$ cannot see, the family's second-order motion
orthogonal to the score
\citep{Efron1975-ic,KassVos1997}. (Efron writes $\gamma$ for
this quantity; we write $\kappa$ to avoid collision with Bahadur's score
$\gamma^{(1)}$.) It is a purely geometric measure, and to attach it to
anything statistical takes asymptotics: only in the limit does $\kappa$
become the information the maximum likelihood estimator loses. Efron's
version of ``Fisher's fundamental theorem'' supplies that link---for iid
sampling with per-observation information $I_{1}$ and curvature
$\kappa_{1}$, the information deficit of the maximum likelihood estimator
converges to a constant,
$nI_{1}-I_{\hat{\theta}_{n}}\rightarrow\kappa_{1}^{2}I_{1}$, which the
classical account reads, through the reciprocal-variance meaning of
information, as the estimator discarding in the limit the information in
about $\kappa_{1}^{2}$ observations.

The efficiency of Definition~\ref{def:Eff} says as much at every finite
sample, and without the detour through asymptotics. Writing
$\tau_{\hat{\theta}_{n}}$ for the score generalized estimator of the
maximum likelihood estimator, Corollary~\ref{cor:eff-corr} gives
$\mathrm{Eff}(\tau_{\hat{\theta}_{n}})
=\mathrm{Corr}^{2}(\tau_{\hat{\theta}_{n}},\gamma^{(1)})
=I_{\hat{\theta}_{n}}/(nI_{1})$, so that
\[
1-\mathrm{Eff}(\tau_{\hat{\theta}_{n}})
=1-\mathrm{Corr}^{2}\bigl(\tau_{\hat{\theta}_{n}},\gamma^{(1)}\bigr)
\longrightarrow\frac{\kappa_{1}^{2}}{n},
\]
the angle between the full score and the marginal score of the maximum
likelihood estimator closing at rate $\kappa_{1}/\sqrt{n}$, its sine the
curvature $\kappa_{n}=\kappa_{1}/\sqrt{n}$ of the $n$-sample family. The
average slope of Section~\ref{subsec:Lambda} thus delivers at every
finite $n$, computed under the family, the statistical content that the
curvature reaches only in the limit and only through a variance---a
pointwise quantity standing in for a family-aware one.

\subsection{Lecture 18: UMVUEs and the
odds}\label{subsec:odds}

Item 15 applies to both remaining plans of
Section~\ref{subsec:example2}. Under plan (b) (negative binomial) the
unbiased theory is narrow: the Cram\'{e}r--Rao bound is attained only
for estimands of the form $a+b/\theta$, and the UMVUE of $\theta$ is
$(k-1)/(N-1)$. The fuller story is under plan (a) (binomial sampling,
$N=n$ fixed), where item 15 applies with $T=\sum X_{i}$: every
square-integrable function of $\bar{X}$ is the UMVUE of its expected
value, and Blackwellization computes them. Bahadur's examples: for
$g(\theta)=\theta^{2}$, starting from $t=I\{X_{1}=X_{2}=1\}$,
\[
\tilde{t}=E_{\theta}(t\mid T)=\frac{T(T-1)}{n(n-1)},
\]
and from it the UMVUE of
$\sigma^{2}(\theta)=\mathrm{Var}_{\theta}(\bar{X})$. Unbiased
estimation here works exactly as intended, and (his Homework 4)
$U_{g}$ is non-empty precisely when $g$ is a polynomial of degree at
most $n$.

The odds, $g(\theta)=\theta/(1-\theta)$, is not a
polynomial: \emph{no unbiased estimate exists}, and the classical
program ends at question (i) with the answer no. The maximum
likelihood plug-in $\hat{t}=\bar{X}/(1-\bar{X})$ is no help: at
$T=n$, a sample of probability $\theta^{n}>0$, it takes the value
$+\infty$, so $E_{\theta}(\hat{t})=+\infty$ for every $\theta$ and
$\hat{t}$ fails Definition~\ref{def:estimate}---it belongs to no
$V_{\theta}$ at all. Bahadur's remedy is
asymptotic---for large $n$ the offending term is small with high
probability.

The generalized account requires no remedy, because nothing has gone
wrong.  The standardized score is
parameterization-invariant (Section~\ref{subsec:modification}), so
inference about the odds $\xi=\theta/(1-\theta)$ \emph{is} inference about $\theta$ relabeled:
the question that has no unbiased answer and the question that has a
complete classical theory are the same question about the same family,
distinguished only by the labels on $M$. At the boundary sample
$T=0$---where the plug-in for $\theta$ exists but equals the boundary
value $0\notin\Theta$, and the plug-in for $\xi$ is degenerate---the
standardized score estimate of Section~\ref{subsec:example2},
rewritten in the odds, is
\begin{equation}
\bar{\tau}_{T=0}(\xi)=-\sqrt{n\theta/(1-\theta)}=-\sqrt{n\xi},
\label{eq:odds-boundary}
\end{equation}
a smooth, strictly decreasing function on $(0,\infty)$: the entire
inferential content of the boundary sample, for the odds, is
the curve $-\sqrt{n\xi}$. The interval
$\{\xi:|\bar{\tau}_{T=0}(\xi)|\le z\}=(0,z^{2}/n]$ is immediate, and
intervals at interior samples come the same way
\citep[Sec.~4.2]{VosWu2025}. 

Where the classical theory must choose
between a constraint that cannot be met ($U_{g}$ empty) and an
estimator that fails its own criterion
($E_{\theta}(\hat{t})=+\infty$), the generalized theory does not
notice that anything is special about the odds. The common
remedy of closing the parameter space---appending $\theta=0,1$ so
that boundary point estimates exist---appends distributions the model
never intended, breaks mutual absolute continuity, and still leaves
$g(\theta)=\theta/(1-\theta)$ undefined at one endpoint; it is a
repair demanded by the insistence that estimates be points of
$\Theta$, and it is not needed once they are not.

\section{The maximum likelihood heuristics}\label{sec:heuristics}

\subsection{Lecture 14: the heuristics}\label{subsec:L14}

Between the information inequality and its attainment theory, Bahadur
inserts a short lecture of a different character, four statements he
is careful to label heuristics:
\begin{enumerate}
\item[(i)]
$W_{\theta}=\mathrm{Span}\{1,\gamma_{\theta}^{(1)},
\gamma_{\theta}^{(2)},\ldots\}$;
\item[(ii)] $W_{\theta}\approx\mathrm{Span}\{1,\gamma_{\theta}^{(1)}\}$
if $s$ is highly informative;
\item[(iii)] the maximum likelihood estimate satisfies
$\hat{\theta}(s)\stackrel{.}{\in}W_{\theta}$---``whatever $\theta$
may be!'';
\item[(iv)] hence, $\hat{\theta}$ is approximately the UMVUE of its
own expected value function.
\end{enumerate}
The supporting calculation expands the likelihood
equation about the truth: writing $L_{\theta}=\log\ell_{\theta}$ and
assuming $I$ large,
\begin{equation}
\hat{\theta}\;\approx\;\theta-\frac{L_{\theta}'}{L_{\theta}''}
\;\approx\;\theta+\frac{1}{\sqrt{I}}\,
\frac{\gamma_{\theta}^{(1)}}{\|\gamma_{\theta}^{(1)}\|},\label{eq:mle-heuristic}
\end{equation}
using $E_{\theta}(L_{\theta}')=0$,
$\mathrm{Var}_{\theta}(L_{\theta}')=I$, and
$-L_{\theta}''/I\approx1$ (item 13). The exclamation in (iii) is
Bahadur's own, and it marks the heuristics' surprising claim: a fixed
function of $s$ is claimed to lie, approximately, in a subspace that
moves with $\theta$.

\subsection{Lectures 25--27: using the score}\label{subsec:L25}

Chapter 7 puts the heuristics to work. If $u(s)$ is any estimate
accurate to order $1/\sqrt{I}$, the Newton--Raphson step for solving
$L'(\theta\mid s)=0$,
\[
u^{*}=u+\Bigl(-\frac{1}{L''(u\mid s)}\Bigr)L'(u\mid s)
\qquad\text{or}\qquad
u^{**}=u+\frac{1}{I(u)}\,L'(u\mid s),
\]
produces estimates with the first-order properties of
$\hat{\theta}$ itself: mean approximately $\theta$ and variance
approximately $1/I$. Bahadur's gallery of location families (normal,
double exponential, Cauchy, and the quartic
$ae^{-bx^{4}}$) shows the score reweighting the observations---equal
weights, central weights, downweighted tails, upweighted
tails---while the quartic example, continued from Lecture 16, shows
$\bar{X}$ failing to be even locally MVUE because the minimal
sufficient statistic  is not complete. Lecture
27 closes the chapter with a review: consistency (hard, and settled
only under compactness via the theorem Bahadur attributes to LeCam
and Wald;
cf.\ \citealp{Wald1949,LeCam1953}), $E_{\theta}(\hat{\theta})\approx\theta$, and
$\mathrm{Var}_{\theta}(\hat{\theta})\approx1/I$, each an approximate
property of the point estimate justified through the exact behavior
of $L'$.

\subsection{The heuristics made exact}\label{subsec:M6}

Every statement in this chapter improves the point estimate by moving
it toward a root of the score, assesses it by its first-order
agreement with the score, and derives its properties from the exact
properties of the score. Read as statements about the point
estimator, the results are approximations; read as statements about
the score, they are exact and already proved. Heuristic (iii), with
its exclamation, dissolves: no fixed function of $s$ lies in every
$W_{\theta}$, but the generalized object $\gamma^{(1)}$ lies in
$W_{\theta}$ for every $\theta$ \emph{by construction}---the moving
subspace is only strange as a home for a non-moving object.
Heuristic (iv) is Theorem~\ref{thm:optimality}: what the point
estimate is approximately (a locally best estimate of its own mean),
the score is exactly (the optimal generalized estimator, uniformly).
The expansion \eqref{eq:mle-heuristic} says the maximum likelihood
estimator is, to first order, a relabeled standardized score; in a
full exponential family the relabeling is exact---the ML estimator of
the mean parameter is $T$ itself, which generates the same
$\sigma$-field as $\hat{\theta}$, so
$\tau_{\hat{\theta}}=\gamma^{(1)}$ in every parameterization---and
in curved families the gap is precisely the information loss
\citet{Efron1975-ic} computes. The Newton iterates are then best
understood as steps toward the score as the inferential object,
rather than toward the point estimate as classically conceived: the
sequence $u,u^{*},u^{**},\ldots$ improves an estimate exactly insofar
as it increases first-order agreement with $\gamma^{(1)}$.

The distinction between generalized estimation and the estimating-function
tradition should be drawn explicitly.
\citet{Godambe1960}, \citet{Godambe1974-pt}, and the
literature following them optimize over unbiased estimating
functions---objects formally close to generalized
estimators---but the estimating function is scaffolding: its role is
to produce a root, and the root is the inference. Here the
generalized estimator \emph{is} the inference: $\tau_{s}$
discriminates among the distributions of the family, its intervals
$\{\theta:|\bar{\tau}_{s}(\theta)|\le z\}$ are read directly from it,
and it retains its meaning at samples (Section~\ref{subsec:odds})
where no usable root exists. The same reversal resolves the two-step
shape of classical practice: find a point estimate, then build
standard errors, tests, and intervals on top of it. The resource the
second step inherits is $I_{t}$, whatever the construction, so
assessing the first step by a pointwise criterion endorses estimators
that cannot support what is to be built on them. A generalized
estimator collapses the two steps---it is already the continuum of
test statistics the second step exists to provide---and consistency,
the one genuinely hard item in Bahadur's closing review, is in this
light a question about extracting a point from the generalized
object, not about the object itself.

\section{The two Hilbert spaces}\label{sec:hilbert}

Bahadur works with one Hilbert space per parameter value:
$V_{\theta}=L^{2}(S,\mathcal{A},P_{\theta})$, with inner product
$(\cdot,\cdot)_{\theta}$, and inside it the subspace $W_{\theta}$
spanned by the likelihood ratios---equal, in the smooth case, to the
span of $\{1,\gamma_{\theta}^{(1)},\gamma_{\theta}^{(2)},\ldots\}$, and in the
exponential family to the square-integrable functions of the
sufficient statistic. Every argument in the lectures is conducted at
a fixed $\theta$: projection onto $W_{\theta}$, orthonormal bases of
$W_{\theta}^{(k)}$, norms and correlations in
$(\cdot,\cdot)_{\theta}$. The family enters each $V_{\theta}$ through
identity \eqref{eq:family-inner}---$W_{\theta}$ is the family as seen
from $\theta$---but the spaces themselves are visited one at a time.

The generalized theory keeps the same functions and reverses the
organization: one space, many inner products. Let
\[
H=\bigcap_{\theta\in\Theta}V_{\theta},
\]
the set of functions square-integrable under every member of the
family. This is not a new construction: by
\eqref{eq:membership} it is exactly the set of Bahadur's estimates,
delivered by his own Definition~\ref{def:estimate}, and mutual
absolute continuity ensures that ``equivalent functions'' means the
same thing at every $\theta$, so the identification of elements
across the $V_{\theta}$ is honest.
On this single space the family of
inner products $\{(\cdot,\cdot)_{\theta}:\theta\in\Theta\}$ does the
work previously distributed over the separate spaces; it is the space
written $H_{M}$ in \citet{VosWu2025}, where the subscript records
that both the space and its inner products belong to the family. A
generalized estimator is then a $\theta$-indexed selection from $H$:
condition (i) of Definition~\ref{def:generalized} places
$\tau(\cdot,\theta)$, for each $\theta$, in the orthogonal complement
of the constants,
\[
H_{\theta}^{\perp}=\{h\in H:E_{\theta}(h)=0\}
=\{h\in H:(h,1)_{\theta}=0\},
\]
and condition (ii), differentiability in $\theta$, is what ties the
selections together---the point at which the relatedness of the
distributions, unused in the classical assessment
(Section~\ref{subsec:formulation}), enters the theory.

The complement of the constants deserves a comment, because it is the
one place the two organizations cut differently. Bahadur's
$W_{\theta}$ \emph{contains} the constants
($\Omega_{\theta,\theta}=1$), and his projections onto $W_{\theta}$
carry the mean along; the estimand appears inside the projection, as
in \eqref{eq:k1-projection}. The generalized theory works in
$H_{\theta}^{\perp}$, where the mean has been removed at every
$\theta$---this is precisely the operation
\eqref{eq:expectation-map} performs on a point estimator,
simultaneously across the family---and what remains of
$W_{\theta}^{(1)}$ after the constants are projected away is the span
of the score alone. The removal is not bookkeeping. Splitting off the
constants at every $\theta$ is the first instance of a general
operation, projecting away the directions associated with quantities
not of inferential interest, and it is the first step in the
treatment of nuisance parameters in the multiparameter theory
\citep{VosWu2025}; the scalar case shows the operation in its
simplest form, with only the mean to remove.

\section{Discussion}\label{sec:discussion}

The limits of the present treatment are deliberate. The parameter is
scalar; the multiparameter
theory, where the orthogonalization of Section~\ref{sec:hilbert}
grows into the treatment of nuisance parameters, is developed in
\citet{VosWu2025}. Bahadur himself points toward that development
from within the omitted chapters: Lecture 23, on metrics in the
parameter space, asks ``Should one use the Euclidean distance
$d_{1}$?'' and answers that ``what is really of interest is the
`distance' between $P_{\delta^{(1)}}$ and
$P_{\delta^{(2)}}$''---offering the Hellinger distance among the
alternatives. Distances between distributions rather than between
labels: the parameterization-invariance that runs through the
present paper is the same instinct, and the geometry of
\citet{VosWu2025} is its systematic development. Asymptotics
(Bahadur's Chapter 8) has been left
out entirely: the generalized statements here are exact at every
sample size, and the one item in Bahadur's closing review that
remains genuinely hard---consistency---concerns the extraction of a
point from the generalized object rather than the object itself.
Families too rough for condition (ii) are outside the theory as
presented, as they are outside Bahadur's smooth chapters; what
replaces $\Lambda$ there is open.

Within the class of generalized estimators the criterion is not
silent: the score is optimal, and the multiplicity in the equality
case of Theorem~\ref{thm:optimality}---any $A(\theta)\gamma^{(1)}$
attains the bound---is the geometry agreeing with the statistics. A
generalized estimate at $\theta_{0}$ is a test statistic, and test
statistics are assessed through tail areas, the significance level
at $\theta_{0}$ and the power elsewhere; increasing transformations
preserve rejection regions and so are statistically equivalent. The
rescalings that make up the score's direction class are exactly such
transformations, applied smoothly in $\theta$: the geometry declares
equivalent precisely what the statistics cannot distinguish.

Where the geometry \emph{is} silent lies one level up, and it is
what this paper leaves open. In the scalar case a generalized
estimator has a potential function $G$ with derivative $\tau$, and
inference can be conducted from $G$ directly \citep{VosWu2025}; the
archetype is the log likelihood, standardized so that its supremum
is zero. Values far from zero again mark models less consistent
with the sample, but the departure is one-sided and the mean of $G$
is not zero---and even recentered, distance from the mean of a
potential is not the statistically relevant quantity. Potential
functions can be compared through the geometry of their generalized
estimators, taking $\Lambda(G)$ to be $\Lambda$ of its derivative,
and on that assessment the log likelihood is optimal among
potential-based procedures, since its derivative is the score. But
the geometry does not distinguish the statistical behavior of a
potential-function procedure from that of its generalized
estimator: the likelihood-ratio and score procedures can differ at
a given sample while their $\Lambda$ assessments coincide. A
criterion that separates a potential from its derivative is the
natural next question.

Within these limits, nothing classical has been overturned. Every
theorem in the lectures
remains a theorem, and the deep apparatus---admissibility,
minimaxity, Bayes averaging, the unbiasedness program of questions
(i)--(iii)---remains what it was. What the impossibility lemma adds
is an explanation of \emph{why} that apparatus exists: each piece is
a forced response to assessing estimators by criteria that use the
family one distribution at a time. Bahadur's own asides mark every
joint at which the response strains, and it is a measure of the
lectures' honesty that the present paper could be organized around
them. The reading offered here is that the strain was never about
estimation; it was about the object asked to carry it.

What the modest modification gains was tracked lecture by lecture
through Bahadur's own development, and bears collecting: existence
at samples where point estimates fail, with the odds at the
boundary sample reduced to the curve $-\sqrt{n\xi}$
(Section~\ref{subsec:odds}); a uniform optimality theorem with a
three-line proof (Theorem~\ref{thm:optimality}) in a subject where,
by Lemma~\ref{lem:pointwise}, no pointwise criterion admits a uniform
optimum at all; assessment that compares like with like, the
information utilized by an estimator against the information in the
family, with no unbiasedness constraint and no closure of the
parameter space; the family-dependence of inference made explicit
(Section~\ref{subsec:example2}); the locally best
unbiased estimates, whose $\theta$-dependence is the classical
theory's ``problem in practice,'' reassembled into a single fully
efficient object (Sections \ref{subsec:M3} and \ref{subsec:Lambda});
attainment and sufficiency recovered as the equality cases of one
bound under the two maps (Section~\ref{sec:attainment}); and the
maximum likelihood heuristics converted from approximate statements
about a point estimator into exact statements about the score
(Section~\ref{sec:heuristics}). All of it follows from one change:
an estimator is a function $\tau$ on $S\times\Theta$ with
$E\tau(s,\theta)=0$, rather than a function $t$ on $S$ alone.

The lectures close their central chapter by asking when the locally
best estimate can be freed of its dependence on $\theta$. The answer
offered here is that it never needed to be freed. What the lens of
point estimation views as the problem in practice is, in Fisher's
view of estimation as a continuum of significance tests, the
solution.

\section*{Acknowledgments}
During the preparation of this manuscript, the author used Claude (Anthropic)
to assist in drafting and revising the text, organizing the exposition,
checking and formatting the references, and preparing the \LaTeX{} source.
The author has reviewed and edited the output, has verified the mathematical
and statistical content, and takes full responsibility for the content of
this publication.

\bibliographystyle{plainnat}
\bibliography{m26a_refs}

\end{document}